\documentclass[12pt,reqno]{amsart}
\usepackage{amsmath,amssymb}
 \makeatletter
 \evensidemargin  \oddsidemargin
 \makeatother
\newtheorem{theorem}{Theorem}
\theoremstyle{plain}

\newtheorem{example}{Example}

\newtheorem{remark}{Remark}

\numberwithin{equation}{section}

 \numberwithin{equation}{section}

\begin{document}
\title[Refined Heinz operator inequalities]{Refined Heinz operator inequalities and norm inequalities}

\author[A.G. Ghazanfari]{A. G. Ghazanfari}

\address{Department of Mathematics, Lorestan University, P.O. Box 465, Khoramabad, Iran.}

\email{ghazanfari.a@lu.ac.ir}

\subjclass[2010]{47A30, 47A63, 15A45.}

\keywords{Norm inequality, Operator inequality, Heinz mean.}

\begin{abstract}
In this article we study the Heinz and Hermite-Hadamard inequalities. We derive the
whole series of refinements of these inequalities involving unitarily invariant norms, which improve some recent results, known from the literature.

We also prove that if $A , B, X\in M_n(\mathbb{C})$ such that $A$ and $B$ are positive definite
and $f$ is an operator monotone function on $(0,\infty)$. Then
\begin{equation*}
|||f(A)X-Xf(B)|||\leq \max\{||f'(A)||, ||f'(B)||\} |||AX-XB|||.
\end{equation*}
Finally we obtain a series of refinements of the Heinz
operator inequalities, which were proved by Kittaneh and Krni\'c.

\end{abstract}
\maketitle

\section{Introduction and preliminaries}

Let $M_{m,n}(\mathbb{C})$ be the space of $m\times n$ complex matrices and $M_n(\mathbb{C})=M_{n,n}(\mathbb{C})$. Let $|||.|||$ denote any unitarily invariant
norm on $M_n(\mathbb{C})$. So,
$|||UAV|||=|||A|||$ for all $A\in M_n(\mathbb{C})$  and for all unitary matrices $U,V\in M_n(\mathbb{C})$. The Hilbert-Schmidt and trace class norm
of $A=[a_{ij}]\in M_n(\mathbb{C})$ are denoted by
\begin{equation*}
\|A\|_2=\left(\sum_{j=1}^n s_j^2(A)\right)^{\frac{1}{2}},\left\| A\right\|_1=\sum_{j=1}^n s_j(A)
\end{equation*}
 where $s_1(A)\geq s_2(A)\geq...\geq s_n(A)$ are the singular values of $A$, which are the eigenvalues of the positive semidefinite
 matrix $\mid A\mid =(A^*A)^{\frac{1}{2}}$, arranged in decreasing order and repeated according to multiplicity.
For Hermitian matrices $A,B\in M_n(\mathbb{C})$, we write that $A\geqslant  0$
if $A$ is positive semidefinite, $A>0$ if $A$ is positive definite, and $A\geqslant B$ if $A-B\geqslant 0$.

The Heron means introduced by Bhatia in \cite{bha} as follows:
\begin{equation*}
K_{\nu}(a,b)=(1-\nu)\sqrt{ab}+\nu \frac{a+b}{2}, \ \ \  0\leq \nu \leq 1.
\end{equation*}
Bhatia derived the inequality
\begin{equation*}
H_{\nu}(a,b)\leq K_{\alpha(\nu)}(a,b),
\end{equation*}
where $\alpha(\nu)=1-4(\nu-\nu ^2)$.

The another one of means that interpolates between
the geometric and the arithmetic means is the logarithmic mean:

\begin{equation*}
 L(a,b)=\int_0^1a^\nu b^{1-\nu} d\nu.
\end{equation*}

Drissi in \cite{dri} showed that $\frac{\sqrt{3}-1}{2\sqrt{3}}\leq \nu\leq \frac{\sqrt{3}+1}{2\sqrt{3}}$ if and only if
\begin{equation}\label{1.2}
H_{\nu}(a,b)\leq L(a,b).
\end{equation}

R. Kaur and M. Singh \cite{kau} have proved that for $A,B,X \in M_{n}$, such that $A,B$ are positive definite,
then for any unitarily invariant norm $|||.|||$, and $\frac{1}{4}\leq \nu \leq \frac{3}{4}$ and $\alpha \in [\frac{1}{2},\infty)$,the following inequality holds
\begin{equation}\label{1.3}
\frac{1}{2}|||A^\nu XB^{1-\nu}+A^{1-\nu}XB^\nu|||\leq \left|\left|\left|(1-\alpha)A^{\frac{1}{2}}XB^{\frac{1}{2}}+\alpha\left(\frac{AX+XB}{2}\right)\right|\right|\right|.
\end{equation}
They also proved the following result:
\begin{align}\label{1.4}
|||A^{\frac{1}{2}}XB^{\frac{1}{2}}|||
&\leq \frac{1}{2}|||A^{\frac{2}{3}}XB^{\frac{1}{3}}+A^{\frac{1}{3}}XB^{\frac{2}{3}}|||\notag\\
&\leq \frac{1}{2+t}|||AX+XB+tA^{\frac{1}{2}}XB^{\frac{1}{2}}|||,
\end{align}
where $A,B,X\in M_{n}$, $A,B$ are positive definite and $-2<t\leq 2$.\\
Obviously, if $A,B,X\in M_{n}$, such that $A,B$ are positive definite, then for $\frac{1}{4}\leq \nu \leq \frac{3}{4}$ and $\alpha \in [\frac{1}{2},\infty)$, and any unitarily invariant norm $|||.|||$, the following inequalities hold
\begin{align}\label{1.5}
|||A^{\frac{1}{2}}XB^{\frac{1}{2}}|||
&\leq \frac{1}{2}|||A^{\nu}XB^{1-\nu}+A^{1-\nu}XB^{\nu}|||\notag\\
&\leq \left|\left|\left|(1-\alpha)A^{\frac{1}{2}}XB^{\frac{1}{2}}+\alpha\left(\frac{AX+XB}{2}\right)\right|\right|\right|,
\end{align}
Suppose that
\begin{equation*}
g_\circ(\nu)=\left|\left|\left|\frac{A^\nu XB^{1-\nu}+A^{1-\nu}XB^\nu}{2}\right|\right|\right|,
\end{equation*}
and
\begin{equation*}
f_\circ(\alpha)=\left|\left|\left|(1-\alpha)A^{\frac{1}{2}}XB^{\frac{1}{2}}+\alpha\left(\frac{AX+XB}{2}\right)\right|\right|\right|.
\end{equation*}
Then, the inequalities \eqref{1.3}, \eqref{1.4},\eqref{1.5}, can be simply rewritten respectively as follows
\begin{align}\label{1.6}
&g_\circ(\nu)\leq f_\circ(\alpha),\notag\\
&g_\circ\left(\frac{1}{2}\right)\leq g_\circ\left(\frac{2}{3}\right)\leq f_\circ\left(\frac{2}{2+t}\right),\\
&g_\circ\left(\frac{1}{2}\right)\leq g_\circ(\nu)\leq f_\circ(\alpha),\notag
\end{align}

I. Ali, H. Yang and A. shakoor \cite{ali} gave a refinement of the inequality \eqref{1.5} as follows:
\begin{equation}\label{1.7}
g_\circ(\nu)\leq (4r_{0}-1)g_\circ\left(\frac{1}{2}\right)+2(1-2r_{0})f_\circ(\alpha),
\end{equation}
where $\frac{1}{4}\leq \nu \leq \frac{3}{4}$, $\alpha \in [\frac{1}{2},\infty )$ and $r_{0}=\min\{\nu,1-\nu\}$.

Kittaneh \cite{kit}, gave a generalization of the Heinz inequality using convexity and
the Hermite–Hadamard integral inequality for $0\leq\nu\leq1$, as follows:
\begin{align}\label{1.8}
2|||A^{\frac{1}{2}}XB^{\frac{1}{2}}|||
&\leq \frac{1}{|1-2\nu|}\left|\int_\nu^{1-\nu}|||A^{t}XB^{1-t}+A^{1-t}XB^{t}|||dt\right|\notag\\
&\leq |||A^{\nu}XB^{1-\nu}+A^{1-\nu}XB^{\nu}|||,
\end{align}
A refinement of \eqref{1.8} is given in \cite{kau2}. They also proved that
\begin{align}\label{1.9}
&\left|\left|\left|A^{\frac{\alpha+\beta}{2}}XB^{1-\frac{\alpha+\beta}{2}}+A^{1-\frac{\alpha+\beta}{2}}XB^{\frac{\alpha+\beta}{2}}\right|\right|\right|\notag\\
&\leq\frac{1}{|\beta-\alpha|}\left|\left|\left|\int_{\alpha}^{\beta}(A^\nu XB^{1-\nu}+A^{1-\nu}XB^\nu)d\nu\right|\right|\right|\notag\\
&=\frac{1}{2}\left|\left|\left|A^{\alpha}XB^{1-\alpha}+A^{1-\alpha}XB^{\alpha}+A^{\beta}XB^{1-\beta}+A^{1-\beta}XB^{\beta}\right|\right|\right|.
\end{align}

Heretofore the inequalities discussed above are proved in the setting of  matrices.
Kapil and Singh in \cite{kap2}, using the contractive maps proved that the relation \eqref{1.9} holds for invertible positive operators in $B(H)$.
The aim of this paper is to obtain refinements of the Hermite–Hadamard inequality \eqref{1.9} in the setting of operators (see Theorem \eqref{t2}).
We also present a generalization of the difference version of Heinz inequality (see Theorem \eqref{t1}). At the end, we study the Heinz
operator inequalities, which were proved in \cite{kit} and give a series of refinements of these operator inequalities (see Theorem \eqref{t4} and \eqref{t5}).

\section{norm inequalities for Matrices}

Let $A , B, X\in M_n(\mathbb{C})$ such that $A$ and $B$ be positive definite and $0\leq\nu\leq1$.
A difference version of the Heinz inequality
\begin{equation}\label{2.0.0}
|||A^\nu XB^{1-\nu}-A^{1-\nu}XB^\nu|||\leq |2\nu-1|~|||AX-XB|||
\end{equation}
was proved by Bhatia and Davis in \cite{bah3}.

Kapil, et.al.,\cite{kap1} proved that if $0<r\leq1$. Then
\begin{equation}\label{2.0}
|||A^rX-XB^r|||\leq r\max\{||A^{r-1}||, ||B^{r-1}||\} |||AX-XB|||.
\end{equation}
They also proved that if $\alpha\geq1$, and $\frac{1-\alpha}{2}\leq\nu\leq\frac{1+\alpha}{2}$, then
\begin{align}\label{2.1}
\alpha |||A^\nu XB^{1-\nu}&-A^{1-\nu}XB^\nu|||\notag\\
&\leq |2\nu-1|\max\{||A^{1-\alpha}||, ||B^{1-\alpha}||\} |||A^\alpha X-XB^\alpha|||.
\end{align}

The following theorem is a generalization of \eqref{2.0}.
\begin{theorem}\label{t1}
Let $A , B, X\in M_n(\mathbb{C})$ such that $A$ and $B$ be positive definite and $f$ be an operator monotone function on $(0,\infty)$. Then
\begin{equation}\label{2.2}
|||f(A)X-Xf(B)|||\leq \max\{||f'(A)||, ||f'(B)||\} |||AX-XB|||.
\end{equation}
\end{theorem}

\begin{proof}
It suffices to prove the required inequality in the special case which
$A = B$ and $A$ is diagonal. Then the general case follows by replacing $A$ with
$\begin{bmatrix}
  A & 0 \\
  0 & B \\
\end{bmatrix}$ and $X$ with
$\begin{bmatrix}
  0 & X \\
  0 & 0 \\
\end{bmatrix}.$
Therefore let $A=diag(\lambda_i)>0$. Then $f(A)X-Xf(A)=Y\circ (AX-XA)$ where $Y=f^{[1]}(A)$, i.e.,

\begin{equation*}
y_{ij}= \\
\begin{cases}
\frac{f(\lambda_i)-f(\lambda_j)}{\lambda_i-\lambda_j}, &\lambda_i\neq \lambda_j\\
f'(\lambda_i),&\lambda_i=\lambda_j.
\end{cases}
\end{equation*}

By \cite[Theorem V.3.4]{bha1}, $f^{[1]}(A)\geq0$.
Consequently
\begin{align*}
|||f(A)X-Xf(A)|||&=|||Y\circ (AX-XA)|||\leq \max y_{ii}~|||AX-XA|||\\
&=||f'(A)||~|||AX-XA|||.
\end{align*}
\end{proof}

\begin{example}
(i) For the function $f(t)=t^r,~0<r<1$,
\begin{align*}
\left|\left|\left|A^rX-XB^r\right|\right|\right|&\leq r\left(\max\{\|A^{r-1}\|, \|B^{r-1}\|\}\right)\\
&=r\left(\max\{\|A^{-1}\|, \|B^{-1}\|\}\right)^{1-r}|||AX-XB|||.
\end{align*}
(ii) For the function $f(t)=\log t$ on $(0,\infty)$,
\begin{equation*}
\left|\left|\left|\log(A)X-X\log(B)\right|\right|\right|\leq
\left(\max\{\|A^{-1}\|, \|B^{-1}\|\}\right)|||AX-XB|||.
\end{equation*}
\end{example}

\begin{remark}
Let $\alpha\geq1$ and $0\leq\nu\leq1$. From inequality \eqref{2.2} for $A^\alpha, B^\alpha$ and $f(t)=t^\frac{1}{\alpha}$, we get
\begin{equation}\label{2.2.0}
|||AX-XB|||\leq \frac{1}{\alpha}\max\{||A^{1-\alpha}||, ||B^{1-\alpha}||\} |||A^\alpha X-XB^\alpha |||.
\end{equation}
On combining \eqref{2.0.0}, and \eqref{2.2.0}, we obtain \eqref{2.1}.
\end{remark}

\section{ Norm inequalities for operators}

Let $B(H)$ denote the  set of all bounded linear operators on a complex
Hilbert space $H$. An operator $A\in B(H)$
is positive, and we write $A\geq 0$, if $(Ax,x)\geq0$ for every vector $x\in H$. If $A$ and $B$ are self-adjoint
operators, the order relation $A\geq B$ means, as usual, that $A-B$ is a positive operator.

To reach inequalities for bounded self-adjoint operators on Hilbert space, we shall use
the following monotonicity property for operator functions:\\
If $X\in B(H)$ is self adjoint with a spectrum $Sp(X)$, and $f,g$  are continuous real valued functions
on an interval containing $Sp(X)$, then
\begin{equation}\label{3.0}
f(t)\geq g(t),~t\in Sp(X)\Rightarrow ~f(X)\geq g(X).
\end{equation}
For more details about this property, the reader is referred to
\cite{pec}.

Let $L_X, R_Y$ denote the left and right multiplication maps on $B(H)$, respectively, that is,
$L_X(T)=XT$ and $R_Y(T)=TY$. Since $L_X$ and $R_Y$ commute, we have
\[
e^{L_X+R_Y}(T)=e^XTe^Y.
\]

Let $U$ be an invertible positive operator in $B(H)$, then there exists a self-adjoint operator $V \in B(H)$ such that $U=e^V$.
Let $n\in \mathbb{N}$ and $A, B$ be two invertible positive operators in $B(H)$. To simplify computations,
we denote $A$ and $B$ by $e^{2^{n+1}X_1}$ and $e^{2^{n+1}Y_1}$, respectively,
where $X_1$ and $Y_1$ in $B(H)$ are self-adjoint. The corresponding operator map $L_{X_{1}}-R_{Y_{1}}$ is denoted by $D$.
With these notations, we now use the results proved in \cite{kap2, lar} to derive the Hermite–Hadamard type inequalities
for unitarily invariant norms.

The Hermite–Hadamard inequality and various refinements
of it in the setting of operators (resp. matrices) were given in \cite{kap2} (resp. \cite{kau2}).
The following theorem is another generalization of the Hermite–Hadamard inequality for operators.

\begin{theorem}\label{t2}
Let $A , B, X\in B(H)$ such that $A$ and $B$ be invertible positive operators and let $\alpha, \beta$
be any two real numbers and $n,m\in \mathbb{N}$. Let $\gamma(t)=(1-t)\alpha+t\beta$,
\begin{align*}
E_n=\frac{1}{2^{n-1}}\sum_{i=1}^{2^{n-1}} \left(A^{\gamma(\frac{2i-1}{2^n})}X B^{1-\gamma(\frac{2i-1}{2^n})}
+A^{1-\gamma(\frac{2i-1}{2^n})}X B^{\gamma(\frac{2i-1}{2^n})}\right),
\end{align*}
and
\begin{align*}
F_m=\frac{1}{2^{m}}\sum_{i=1}^{2^{m-1}}\left(A^{\gamma(\frac{i-1}{2^{m-1}})}X B^{1-\gamma(\frac{i-1}{2^{m-1}})}
+A^{1-\gamma(\frac{i-1}{2^{m-1}})}X B^{\gamma(\frac{i-1}{2^{m-1}})}\right.\\
\left.+A^{\gamma(\frac{i}{2^{m-1}})}X B^{1-\gamma(\frac{i}{2^{m-1}})}+A^{1-\gamma(\frac{i}{2^{m-1}})}X B^{\gamma(\frac{i}{2^{m-1}})}\right).
\end{align*}

Then
\begin{align}\label{2.2.1}
&\left|\left|\left|A^{\frac{\alpha+\beta}{2}}XB^{1-\frac{\alpha+\beta}{2}}
+A^{1-\frac{\alpha+\beta}{2}}XB^{\frac{\alpha+\beta}{2}}\right|\right|\right|=|||E_1|||\leq \dots \leq |||E_n|||\notag\\
&\leq\frac{1}{|\beta-\alpha|}\left|\left|\left|\int_{\alpha}^{\beta}(A^\nu XB^{1-\nu}+A^{1-\nu}XB^\nu)d\nu\right|\right|\right|\notag\\
&\leq |||F_m|||\leq \dots \leq |||F_1|||\notag\\
&=\frac{1}{2}\left|\left|\left|A^{\alpha}XB^{1-\alpha}+A^{1-\alpha}XB^{\alpha}+A^{\beta}XB^{1-\beta}+A^{1-\beta}XB^{\beta}\right|\right|\right|.
\end{align}
\end{theorem}

\begin{proof}
Put $A=e^{2^{n+1}X_1}$, $B=e^{2^{n+1}Y_1}$ and $T=A^\frac{1}{2}XB^{\frac{1}{2}}$, then
\begin{align*}\vspace{.2cm}
A^{\gamma(\frac{2i-1}{2^n})}X B^{1-\gamma(\frac{2i-1}{2^n})}&+A^{1-\gamma(\frac{2i-1}{2^n})}X B^{\gamma(\frac{2i-1}{2^n})}\\
&=2\cosh\left(2^{n+1}\left(\gamma\Big(\frac{2i-1}{2^{n}}\Big)-\frac{1}{2}\right)D\right)T.
\end{align*}
Similarly, a simple calculation shows
\begin{align*}
&A^{\gamma(\frac{i-1}{2^{n-1}})}X B^{1-\gamma(\frac{i-1}{2^{n-1}})}+A^{1-\gamma(\frac{i-1}{2^{n-1}})}X B^{\gamma(\frac{i-1}{2^{n-1}})}\\
&+A^{\gamma(\frac{i}{2^{n-1}})}X B^{1-\gamma(\frac{i}{2^{n-1}})}+A^{1-\gamma(\frac{i}{2^{n-1}})}X B^{\gamma(\frac{i}{2^{n-1}})}\\
&=2\cosh\left(2^{n}\left(\gamma\Big(\frac{i-1}{2^{n-1}}\Big)-\frac{1}{2}\right)D\right)T+2\cosh\left(2^{n}\left(\gamma\Big(\frac{i}{2^{n-1}}\Big)
-\frac{1}{2}\right)D\right)T.
\end{align*}
Continuing the calculation, we have
\begin{align*}
A^{\gamma(\frac{i-1}{2^{n-1}})}&X B^{1-\gamma(\frac{i-1}{2^{n-1}})}+A^{1-\gamma(\frac{i-1}{2^{n-1}})}X B^{\gamma(\frac{i-1}{2^{n-1}})}\\
&=4\cosh\left(2^{n-1}\left(\gamma\Big(\frac{i-1}{2^{n-1}}\Big)+\gamma\Big(\frac{i}{2^{n-1}}\Big)-1\right)D\right)\\
&\times\cosh\left(2^{n-1}\left(\gamma\Big(\frac{i-1}{2^{n-1}}\Big)-\gamma\Big(\frac{i}{2^{n-1}}\Big)\right)D\right)T\\
&=4\cosh\left(2^{n-1}\left(\gamma\Big(\frac{i-1}{2^{n-1}}\Big)+\gamma\Big(\frac{i}{2^{n-1}}\Big)-1\right)D\right)\\
&\times \cosh((\beta-\alpha)D)T,
\end{align*}
and
\begin{align*}
\frac{2^n}{\beta-\alpha}\int_{\gamma(\frac{i-1}{2^{n}})}^{\gamma(\frac{i}{2^{n}})}&(A^\nu XB^{1-\nu}+A^{1-\nu}XB^\nu)d\nu\\
&=\frac{2^n}{\beta-\alpha}\int_{\gamma(\frac{i-1}{2^{n}})}^{\gamma(\frac{i}{2^{n}})}
2\cosh\left(2^{n+1}\left(\nu-\frac{1}{2}\right)D\right)Td\nu\\
&=\frac{D^{-1}}{\beta-\alpha}\left[\sinh\left(2^{n+1}\left(\gamma\Big(\frac{i}{2^{n}}\Big)
-\frac{1}{2}\right)D\right)\right.\\
&\left.-\sinh\left(2^{n+1}\left(\gamma\Big(\frac{i-1}{2^{n}}\Big)-\frac{1}{2}\right)D\right)\right]T.
\end{align*}
Consequently,
\begin{align*}
\frac{2^n}{\beta-\alpha}\int_{\gamma(\frac{i-1}{2^{n}})}^{\gamma(\frac{i}{2^{n}})}&(A^\nu XB^{1-\nu}+A^{1-\nu}XB^\nu)d\nu\\
&=\frac{2D^{-1}}{\beta-\alpha}\cosh\left(2^{n}\left(\gamma\Big(\frac{i-1}{2^{n}}\Big)+\gamma\Big(\frac{i}{2^{n}}\Big)-1\right)D\right)\\
&\times\sinh\left(2^{n}\left(\gamma\Big(\frac{i}{2^{n}}\Big)-\gamma\Big(\frac{i-1}{2^{n}}\Big)\right)D\right)T\\
&=\frac{2D^{-1}}{\beta-\alpha}\cosh\left(2^{n}\left(\gamma\Big(\frac{i-1}{2^{n}}\Big)+\gamma\Big(\frac{i}{2^{n}}\Big)-1\right)D\right)\\
&\times\sinh((\beta-\alpha)D)T.
\end{align*}
Calculus computations show that for $n\geq2$, we have

\begin{align*}
E_n&=\frac{1}{2^{n-2}}\sum_{i=1}^{2^{n-1}}\cosh\left(2^{n+1}\left(\gamma\Big(\frac{2i-1}{2^{n}}\Big)-\frac{1}{2}\right)D\right)T\\
&=\frac{1}{2^{n-2}}\left[\sum_{i=1}^{2^{n-2}}\cosh\left(2^{n+1}\left(\gamma\Big(\frac{2i-1}{2^{n}}\Big)-\frac{1}{2}\right)D\right)\right.\\
&\left.+\frac{1}{2^{n-2}}\sum_{i=1+2^{n-2}}^{2^{n-1}}\cosh\left(2^{n+1}\left(\gamma\Big(\frac{2i-1}{2^{n}}\Big)-\frac{1}{2}\right)D\right)\right]T\\
&=\frac{1}{2^{n-2}}\sum_{i=1}^{2^{n-2}}\left[\cosh\left(2^{n+1}\left(\gamma\Big(\frac{2i-1}{2^{n}}\Big)-\frac{1}{2}\right)D\right)\right.\\
&\left.+\cosh\left(2^{n+1}\left(\gamma\Big(1-\frac{2i-1}{2^{n}}\Big)-\frac{1}{2}\right)D\right)\right]T\\
&=\frac{1}{2^{n-3}}\sum_{i=1}^{2^{n-2}}\left[\cosh\left(2^{n}\left(\gamma\Big(\frac{2i-1}{2^{n}}\Big)+\gamma\Big(1-\frac{2i-1}{2^{n}}\Big)-1\right)D\right)\right.\\
&\left.\times \cosh\left(2^{n}\left(\gamma\Big(\frac{2i-1}{2^{n}}\Big)-\gamma\Big(1-\frac{2i-1}{2^{n}}\Big)\right)D\right)\right]T.
\end{align*}
Using the relations $\gamma(t)+\gamma(1-t)=\alpha+\beta$ and $\gamma(t)-\gamma(1-t)=(2t-1)(\beta-\alpha)$, we obtain

\begin{align}\label{2.3}
E_n&=\frac{1}{2^{n-3}}\cosh\left(2^{n}(\alpha+\beta-1)D\right)\sum_{i=1}^{2^{n-2}}\cosh\left(2^{n}\Big(\frac{2i-1}{2^{n-1}}-1\Big)(\beta-\alpha)D\right)T\notag\\
&=\frac{1}{2^{n-3}}\cosh\left(2^{n}(\alpha+\beta-1)D\right)\sum_{i=1}^{2^{n-2}}\cosh\left(2(2i-1)(\beta-\alpha)D\right)T\notag\\
&=2\cosh\left(2^{n}(\alpha+\beta-1)D\right)\prod_{i=1}^{n-1} \cosh\left(2^{n-i}(\beta-\alpha)D\right)T.
\end{align}
Similarly, by simple calculations, we obtain
\begin{align}\label{2.4}
&F_{n+1}=\frac{1}{2^{n-2}}\sum_{i=1}^{2^{n-1}}\cosh\left(2^{n-1}\left(\gamma\Big(\frac{i-1}{2^{n-1}}\Big)
+\gamma\Big(\frac{i}{2^{n-1}}\Big)-1\right)D\right)\cosh((\beta-\alpha)D)T\notag\\
&=\cosh\left(2^{n}(\alpha+\beta-1)D\right)\prod_{i=1}^{n-1} \cosh\left(2^{n-i}(\beta-\alpha)D\right)\Big(\cosh(2(\beta-\alpha)D)+1\Big)T,
\end{align}
and
\begin{align}\label{2.5}
&W:=\frac{1}{\beta-\alpha}\int_{\alpha}^{\beta}(A^\nu XB^{1-\nu}+A^{1-\nu}XB^\nu)d\nu\notag\\
&=\frac{2D^{-1}}{\beta-\alpha}\sum_{i=1}^{2^n}\cosh\left(2^{n}\left(\gamma\Big(\frac{i-1}{2^{n}}\Big)
+\gamma\Big(\frac{i}{2^{n}}\Big)-1\right)D\right)\sinh\left((\beta-\alpha)D\right)T\notag\\
&=\frac{2D^{-1}}{\beta-\alpha}\cosh\left(2^{n}(\alpha+\beta-1)D\right)\prod_{i=1}^n \cosh\left(2^{n-i}(\beta-\alpha)D\right)
\sinh\left((\beta-\alpha)D\right)T\notag\\
&=\frac{D^{-1}}{2^{n-1}(\beta-\alpha)}\cosh\left(2^{n}(\alpha+\beta-1)D\right)\sinh\left(2^{n}(\beta-\alpha)D\right)T.
\end{align}
 By \cite[Proposition 21]{lar}, the operator map $\frac{2(\beta-\alpha)D}{\sinh(2(\beta-\alpha)D)}$ is contractive,
 so from equality \eqref{2.3} and \eqref{2.5}, we obtain
 \begin{equation}\label{2.6}
 |||E_n|||\leq |||W|||.
\end{equation}
From equality \eqref{2.3} for $E_{n-1}$ with $A=e^{2^{n+1}X_1}, B=e^{2^{n+1}Y_1}$, we get
\begin{align*}
E_{n-1}=2\cosh\left(2^{n}(\alpha+\beta-1)D\right)\prod_{i=1}^{n-2} \cosh\left(2^{n-i}(\beta-\alpha)D\right)T.
\end{align*}
The operator map $\frac{1}{\cosh(2(\beta-\alpha)D)}$ is contractive, so
\begin{equation}\label{2.7}
|||E_{n-1}|||\leq |||E_n|||.
\end{equation}

 By \cite[Proposition 2.4]{kap2}, the operator map $\frac{\sinh((\beta-\alpha)D)}{(\beta-\alpha)D\cosh((\beta-\alpha)D)}$ is contractive, therefore
 from equality \eqref{2.4} and \eqref{2.5}, we get
 \begin{equation}\label{2.8}
 |||W|||\leq |||F_{n+1}|||.
\end{equation}

 From equality \eqref{2.5} for $n=2$, i.e., for $A=e^{8X_1}, B=e^{8Y_1}$, we have
 $$W=\frac{D^{-1}}{2(\beta-\alpha)}\cosh\left(4(\alpha+\beta-1)D\right)\sinh\left(4(\beta-\alpha)D\right)T$$
and $$F_2=\cosh\left(4(\alpha+\beta-1)D\right)\Big(\cosh\left(4(\beta-\alpha)D\right)+1\Big)T.$$
In this case, we also get $|||W|||\leq |||F_{2}|||$ because the operator map\\
 $\frac{\sinh(2(\beta-\alpha)D)}{2(\beta-\alpha)D\cosh(2(\beta-\alpha)D)}$ is contractive.

From equality \eqref{2.4} for $F_{n}$ with $A=e^{2^{n+1}X_1}, B=e^{2^{n+1}Y_1}$, we get
\begin{align*}
F_{n}=\cosh\left(2^{n}(\alpha+\beta-1)D\right)\prod_{i=1}^{n-2} &\cosh\left(2^{n-i}(\beta-\alpha)D\right)\\
&\times\Big(\cosh(4(\beta-\alpha)D)+1\Big)T.
\end{align*}
Therefore
\begin{align*}
\frac{F_{n+1}}{F_{n}}&=\frac{\cosh\left(2(\beta-\alpha)D\right)(1+\cosh\left(2(\beta-\alpha)D\right))}{1+\cosh\left(4(\beta-\alpha)D\right)}\\
&=\frac{1}{2}\left(\frac{1}{\cosh(2(\beta-\alpha)D)}+1\right),
\end{align*}
and this implies that
\begin{equation}\label{2.9}
|||F_{n+1}|||\leq |||F_{n}|||.
\end{equation}

From \eqref{2.6}, \eqref{2.7}, \eqref{2.8} and \eqref{2.9}, we obtain the relation \eqref{2.2.1} and the proof is completed.
\end{proof}

\begin{theorem}\label{t3}
Let $A , B, X\in B(H)$ such that $A$ and $B$ be invertible positive operators. Let $\frac{1}{4}\leq\nu\leq\frac{3}{4}$
and $\alpha\in[\frac{1}{2}, \infty)$. Then
\begin{align}\label{2.10}
\frac{1}{2}|||A^\nu XB^{1-\nu}+A^{1-\nu}XB^\nu|||&\leq \left|\left|\left|\int_0^1 A^tXB^{1-t} dt \right|\right|\right|\\
&\leq \left|\left|\left|(1-\alpha)A^{\frac{1}{2}}XB^{\frac{1}{2}}+\alpha\left(\frac{AX+XB}{2}\right)\right|\right|\right|.\notag
\end{align}
\end{theorem}

\begin{proof}
Suppose that $A=e^{2X_1}, B=e^{2Y_1}$ and $T=A^\frac{1}{2}XB^\frac{1}{2}$,then
\begin{align*}
\frac{1}{2}|||A^\nu XB^{1-\nu}+A^{1-\nu}XB^\nu|||=|||\cosh \big((2\nu-1)D\big)T|||,
\end{align*}
and
\begin{align*}
\left|\left|\left|\int_0^1 A^tXB^{1-t} dt \right|\right|\right|=\left|\left|\left|\int_0^1 \exp\big((2t-1)D\big)T dt \right|\right|\right|
=\left|\left|\left| D^{-1}\sinh(D)T\right|\right|\right|.
\end{align*}
By \cite[Proposition 21]{lar}, the operator map $\frac{D\cosh \big((2\nu-1)D\big)}{\sinh(D)}$ is contractive. This proves the first inequality
 in \eqref{2.10}. The second inequality in \eqref{2.10} was proved in Theorem 3.9 of \cite{kap2}.

\end{proof}

\section{Improved Heinz operator inequalities}

Let $A,B\in B(H)$ be two positive operators and $\nu\in [0,1]$, then the
$\nu$-weighted arithmetic mean of $A$ and $B$ denoted by
$A\nabla_{\nu} B$, is defined as $A\nabla_{\nu} B=(1-\nu)A+\nu B$.
If $A$ is invertible, the $\nu$-geometric mean of $A$ and $B$
denoted by $A\sharp_{\nu} B$ is defined as $A\sharp_{\nu}
B=A^{\frac{1}{2}}(A^{\frac{-1}{2}}BA^{\frac{-1}{2}})^{\nu}
A^{\frac{1}{2}}$. For more detail,
see Kubo and Ando \cite {kub}. When $v=\frac{1}{2}$ , we write
$A\nabla B$, $A\sharp B$, for brevity, respectively.

Let $A,B\in B(H)$ be two invertible positive (strictly positive) operators and $\nu\in [0,1]$. The
operator version of the Heinz means are defined by
\begin{equation*}
 H_{\nu}(A,B)=\frac{A\sharp_{\nu} B+A\sharp_{1-\nu} B}{2},
 \end{equation*}
 and the operator version of the Heron means are defined by
\begin{equation*}
 K_\nu(A,B)=(1-\nu)(A\sharp B)+\nu( A\nabla B).
 \end{equation*}
 Zhao et al. in \cite{zha} gave an inequality
for the Heinz-Heron means as follows:\\
$$H_{\nu}(A,B)\leq K_{\alpha(\nu)}(A,B),$$
where $\alpha(\nu)=1-4(\nu-\nu^2)$.\\

It is easy to show that the above Heinz mean $H_\nu(\cdot,\cdot)$ interpolates between the non-weighted
arithmetic mean and geometric mean, that is
\begin{equation}\label{3.1}
A\sharp B\leq H_\nu(A,B)\leq A\nabla B.
\end{equation}

Kittaneh and Krni\'c in \cite{kit2} obtained the some refinements of the left and right inequalities in \eqref{3.1} for $\nu\in[0,1]-\{\frac{1}{2}\}$, as follows:
\begin{align}\label{3.2}
A\sharp B&\leq H_{\frac{2\nu+1}{4}}(A,B)\leq\frac{1}{2\nu-1}A^{\frac{1}{2}}F_\nu(A^{\frac{-1}{2}}BA^{\frac{-1}{2}})A^{\frac{1}{2}}\notag\\
&\leq \frac{1}{4}H_\nu(A,B)+\frac{1}{2}H_{\frac{2\nu+1}{4}}(A,B)+\frac{1}{4}A\nabla B\notag\\
&\leq \frac{1}{2}H_\nu(A,B)+\frac{1}{2}A\sharp B+\leq H_\nu(A,B),
\end{align}
and
\begin{align}\label{3.3}
H_\nu(A,B)&\leq H_{\frac{r_0}{2}}(A,B)\leq \frac{1}{2r_0}A^{\frac{1}{2}}\left[F_1(A^{\frac{-1}{2}}BA^{\frac{-1}{2}})+F_{r_0}(A^{\frac{-1}{2}}BA^{\frac{-1}{2}})\right]A^{\frac{1}{2}}\notag\\
&\leq \frac{1}{4}H_\nu(A,B)+\frac{1}{2}H_{\frac{r_0}{2}}(A,B)+\frac{1}{4}A\nabla B\\
&\leq \frac{1}{2}H_\nu(A,B)+\frac{1}{2}A\nabla B\leq A\nabla B,\notag
\end{align}
where $r_0=\min\{\nu,1-\nu\}$ and

\begin{equation}\label{3.4}
F_\nu(x)= \\
\begin{cases}
\frac{x^\nu -x^{1-\nu}}{\log x}, &x>0, x\neq 1\\
2\nu-1,& x=1.
\end{cases}
\end{equation}

Let $f,\alpha,\beta$ be continuous real functions on $\mathbb{R}$ and $f$ be convex. Let $\alpha(\nu)<\beta(\nu)~(\nu\in\mathbb{R})$, and
$\gamma_\nu(t)=(1-t)\alpha(\nu)+t\beta(\nu)$. For $ n\in \mathbb{N}$, Define
\begin{alignat}{2}\label{3.5}
\varphi_n(f,\nu)&=\frac{1}{2^{n-1}}\sum_{i=1}^{2^{n-1}}f\left(\left(1-\frac{2i-1}{2^n}\right)\alpha(\nu)+\frac{2i-1}{2^n}\beta(\nu)\right)\quad &&(\nu\in\mathbb{R})\notag\\
&=\frac{1}{2^{n-1}}\sum_{i=1}^{2^{n-1}}f\left(\gamma_\nu\left(\frac{2i-1}{2^n}\right)\right).
\end{alignat}
For $m\in\mathbb{N}$, we define
\[
\Phi_1(\nu)=\frac{f(\alpha(\nu))+f(\beta(\nu))}{2},
\]
and for $m\geq 1$
\begin{alignat}{2}\label{3.6}
\Phi_{m+1}(f,\nu)&=\frac{1}{2^{m+1}}\left[f(\alpha(\nu))+f(\beta(\nu))+2\sum_{i=1}^{2^{m}-1}f\left(\left(1-\frac{i}{2^{m}}\right)\alpha(\nu) +\frac{i}{2^{m}}\beta(\nu)\right)\right]\notag\\
&=\frac{1}{2^{m+1}}\left[f(\alpha(\nu))+f(\beta(\nu))+2\sum_{i=1}^{2^{m}-1}f\left(\gamma_\nu\left(\frac{i}{2^{m}}\right)\right)\right].
\end{alignat}
It can be easily shown that for every $n, m\in\mathbb{N}$, the sequence $(\varphi_n),~(\text{resp}.(\Phi_m))$ is an increasing (resp. a decreasing) sequence of continuous functions such that

\begin{align}\label{3.7}
f\left(\frac{\alpha+\beta}{2}\right)\leq\phi_n(f,\nu)\leq \frac{1}{\beta-\alpha}\int_{\alpha}^{\beta} f(t)dt
\leq\Phi_m(f,\nu)\leq\frac{f(\alpha)+f(\beta)}{2}
\end{align}

and
\begin{equation}\label{3.8}
\lim_{n\rightarrow\infty}\varphi_n(f,\nu)=\lim_{m\rightarrow\infty}\Phi_m(f,\nu)=\frac{1}{\beta-\alpha}\int_{\alpha}^{\beta} f(t)dt.
\end{equation}
Now, we consider the function $f_x:[0,1]\rightarrow \mathbb{R}$, $x>0$, by
\begin{equation}\label{3.9}
f_x(t)=\frac{x^t+x^{1-t}}{2},
\end{equation}
and $0\leq\alpha(\nu)<\beta(\nu)\leq1$.
The functions $\varphi_n(f_x,\nu)$ and $\Phi_n(f_x,\nu)$ are continuous functions of $x$.
If $A,B\in B(H)$ are two invertible positive operators, using the functional calculus at $x=A^{\frac{-1}{2}}BA^{\frac{-1}{2}}$ for $\varphi_n(f_x,\nu)$,
we have
 \begin{align}\label{3.9.1}
\varphi_n(f_{A^{\frac{-1}{2}}BA^{\frac{-1}{2}}},\nu)=\frac{1}{2^{n-1}}
\sum_{i=1}^{2^{n-1}}\frac{(A^{\frac{-1}{2}}BA^{\frac{-1}{2}})^{\gamma_\nu\left(\frac{2i-1}{2^n}\right)}
+(A^{\frac{-1}{2}}BA^{\frac{-1}{2}})^{1-\gamma_\nu\left(\frac{2i-1}{2^n}\right)}}{2}.
\end{align}
Multiplying \eqref{3.9.1} by $A^\frac{1}{2}$ on the left and right sides, we get

 \begin{align}\label{3.10}
A^\frac{1}{2}\varphi_n(f_{A^{\frac{-1}{2}}BA^{\frac{-1}{2}}},\nu)A^\frac{1}{2}=\frac{1}{2^{n-1}}\sum_{i=1}^{2^{n-1}}H_{\gamma_\nu\left(\frac{2i-1}{2^n}\right)}(A,B).
\end{align}
We denote it by $\varphi_{n}(\alpha,\beta; A, B)$. Similarly,

\begin{align}\label{3.11}
\Phi_{m+1}(\alpha,\beta; A,B)&:=A^\frac{1}{2}\Phi_{m+1}(f_x,\nu)A^\frac{1}{2}\\
&=\frac{1}{2^{m+1}}\left[H_{\alpha(\nu)}(A,B)+H_{\beta(\nu)}(A,B)+2\sum_{i=1}^{2^{m}-1}H_{\gamma_\nu\left(\frac{i}{2^{m}}\right)}(A,B)\right].\notag
\end{align}

In the following Theorem we give a series of refinements of \eqref{3.2}.

\begin{theorem}\label{t4}
Let $n,m\in\mathbb{N}$ and $n>1, m>2$. If $A,B\in B(H)$ are two invertible positive operators, then the series of  inequalities holds
\begin{align}\label{3.12}
A\sharp B&\leq H_{\frac{2\nu+1}{4}}(A,B)=\varphi_{1}\left(\nu,\frac{1}{2};A,B\right)\leq\varphi_{n}\left(\nu,\frac{1}{2};A,B\right)\notag\\
&\leq\frac{1}{2\nu-1}A^{\frac{1}{2}}F_\nu(A^{\frac{-1}{2}}BA^{\frac{-1}{2}})A^{\frac{1}{2}}\leq \Phi_{m}\left(\nu,\frac{1}{2};A,B\right)\notag\\
&\leq\Phi_{2}\left(\nu,\frac{1}{2};A,B\right)=\frac{1}{4}H_\nu(A,B)+\frac{1}{2}H_{\frac{2\nu+1}{4}}(A,B)+\frac{1}{4}A\sharp B\notag\\
&\leq \frac{1}{2}H_\nu(A,B)+\frac{1}{2}A\sharp B+\leq H_\nu(A,B),
\end{align}
for all $\nu\in[0,1]-\{\frac{1}{2}\}$, where $F_\nu$ is the function given in \eqref{3.4}.

\end{theorem}

\begin{proof}
Let $0\leq \nu<\frac{1}{2}$. Applying inequality \eqref{3.7} to the function $f_x$ and $\alpha(\nu)=\nu, \beta(\nu)=\frac{1}{2}$, we get

\begin{align}\label{3.13}
f_x\left(\frac{2\nu+1}{4}\right)&\leq\phi_n(f_x,\nu)\leq \frac{2}{1-2\nu}\int_{\nu}^{\frac{1}{2}} f(t)dt\notag\\
&\leq\Phi_m(f_x,\nu)\leq\frac{f_x(\nu)+f_x(\frac{1}{2})}{2}.
\end{align}

Clearly, $\varphi_{n}(\alpha,\beta; A, B)=\varphi_{n}(\beta,\alpha; A, B)$ and $\Phi_{m}(\alpha,\beta; A, B)=\Phi_{m}(\beta,\alpha; A, B)$
since $H_{1-\nu}(A,B)=H_\nu(A,B)$. Therefore \eqref{3.13} also holds for $\frac{1}{2}<\nu\leq1$ because $F_{1-\nu}(x)=-F_\nu(x)$.

Utilizing of the monotonicity property \eqref{3.0}, the relation \eqref{3.13} holds when $x$ is
replaced with the positive operator $A^{\frac{-1}{2}}BA^{\frac{1}{2}}$.
Finally, multiplying both sides of such obtained series of inequalities by $A^\frac{1}{2}$ and applying \eqref{3.10}
and \eqref{3.11}, we deduced the inequalities \eqref{3.12}.

\end{proof}

In the following Theorem we give a series of refinements of \eqref{3.3}.

\begin{theorem}\label{t5}
Let $1\leq n,m\in\mathbb{N}$ and $\nu\in[0,1]-\{\frac{1}{2}\}$. If $A,B\in B(H)$ are two invertible positive operators, then the series of  inequalities holds
\begin{align}\label{3.14}
&H_\nu(A,B)\leq H_{\frac{r_0}{2}}(A,B)\leq \varphi_{n}(0,r_0;A,B)\notag\\
&\leq \frac{1}{2r_0}A^{\frac{1}{2}}\left[F_1(A^{\frac{-1}{2}}BA^{\frac{-1}{2}})
+F_{r_0}(A^{\frac{-1}{2}}BA^{\frac{-1}{2}})\right]A^{\frac{1}{2}}\leq\Phi_{m}(0,r_0;A,B)\notag\\
&\leq \frac{1}{4}H_\nu(A,B)+\frac{1}{2}H_{\frac{r_0}{2}}(A,B)+\frac{1}{4}A\nabla B\\
&\leq \frac{1}{2}H_\nu(A,B)+\frac{1}{2}A\nabla B\leq A\nabla B,\notag
\end{align}
where $r_0=\min\{\nu,1-\nu\}$ and $F_\nu$ is the function given in \eqref{3.4}.
\end{theorem}

\begin{proof}
By the symmetry of the Heinz means and the fact that $F_{1-\nu}=-F_\nu$, it is sufficient that, we prove \eqref{3.14} for $0\leq \nu<\frac{1}{2}$.
 Applying inequality \eqref{3.7} to the function $f_x$ and $\alpha(\nu)=0, \beta(\nu)=r_0=\min\{\nu,1-\nu\}=\nu$, we get

\begin{align}\label{3.15}
f_x\left(\frac{\nu}{2}\right)&\leq\phi_n(f_x,\nu)\leq \frac{1}{\nu}\int_{0}^{\nu} f(t)dt\notag\\
&\leq\Phi_m(f_x,\nu)\leq\frac{f_x(0)+f_x(\nu)}{2}.
\end{align}

By the same argument used in the proof of Theorem \ref{t4},
we obtain the inequalities \eqref{3.14}.

\end{proof}

%{\bf Acknowledgments.}  The authors would like to express their thanks to referees
%for careful reading and kind suggestion.

%********************************************************************************************************************************************************

\bibliographystyle{amsplain}

\end{document}